\documentclass[12pt]{amsart}
\usepackage[numbers]{natbib}
\usepackage{graphics}
\usepackage[toc,page]{appendix}
\usepackage{amscd,ragged2e}
\usepackage{amsmath}
\usepackage{amssymb}

\usepackage{graphicx}

\usepackage{epsfig}
\usepackage{graphicx}

\usepackage{dsfont}

\usepackage{amsmath}
\usepackage{amsthm}
\usepackage{amssymb,hyperref,dsfont}

\newtheorem{theorem}{Theorem}[section]
\newtheorem{definition}[theorem]{Definition}
\newtheorem{proposition}[theorem]{Proposition}

\newtheorem{lemma}[theorem]{Lemma}

\makeatletter
\newcommand\incircbin
{%
  \mathpalette\@incircbin
}
\newcommand\@incircbin[2]
{%
  \mathbin%
  {%
    \ooalign{\hidewidth$#1#2$\hidewidth\crcr$#1\bigcirc$}%
  }%
}
\newcommand{\ostar }{\incircbin{\star}}
\makeatother

\begin{document}

\author{J.P. McCarthy}
\address{Department of Mathematics, Cork Institute of Technology, Cork, Ireland.\newline \qquad\emph{E-mail address:} {\tt jeremiah.mccarthy@cit.ie}}

\title[Diaconis--Shahshahani Upper Bound Lemma for Finite Quantum Groups]
{Diaconis--Shahshahani Upper Bound Lemma for Finite Quantum Groups}


\begin{abstract}
A central tool in the study of ergodic random walks on finite groups is the Upper Bound Lemma of Diaconis and Shahshahani. The Upper Bound Lemma uses the representation theory of the group to generate upper bounds for the distance to random and thus can be used to determine convergence rates for ergodic walks. The representation theory of quantum groups is remarkably similar to the representation theory of classical groups. This allows for a generalisation of the Upper Bound Lemma to an Upper Bound Lemma for finite quantum groups.  The Upper Bound Lemma is used to study the convergence of ergodic random walks on the dual group $\widehat{S_n}$ as well as on the truly quantum groups of Sekine.
\end{abstract}

\maketitle

\section*{Introduction}
Let $\sigma_1,\sigma_2,\dots,\sigma_k$ be a sequence of shuffles of a deck of cards. If the deck starts in some known order, the order of the deck after these $k$ shuffles is given by
$$\Sigma_k=\sigma_k\cdots\sigma_2\cdot \sigma_1.$$
Suppose the shuffles are random variables independently and identically distributed as $\sigma_i\sim\nu\in M_p(S_{52})$, the set of probability distributions on $S_{52}$, then $\Sigma_k\sim \nu^{\star k}$ where $\nu^{\star k}$ is defined inductively as
$$\nu^{\star (k+1)}(\{s\})=\sum_{t\in G}\nu(\{st^{-1}\})\nu^{\star k}(\{t\}).$$

 If $\nu$ is not concentrated on a subgroup, nor a coset of a normal subgroup of $S_{52}$, then as $k\rightarrow \infty$, $\nu^{\star k}\rightarrow \pi_{S_{52}}$, the \emph{uniform} or \emph{random} distribution on $S_{52}$.

\bigskip

This generalises to arbitrary finite groups. Given given independent and identically distributed $s_i\sim \nu\in M_p(G)$, not concentrated on a subgroup, nor a coset of a normal subgroup, as $k\rightarrow \infty$:
$$\xi_k=s_k\cdots s_2\cdot s_1\,\sim \,\nu^{\star k}\rightarrow \pi,$$
the random distribution on $G$. In this case the random walk is said to be \emph{ergodic}.

%\bigskip
\newpage
In the previous statement, convergence is with respect to the \emph{total variation distance}:

$$\|\nu^{\star k}-\pi\|:=\sup_{S\subset G}|\nu^{\star k}(S)-\pi(S)|.$$

\bigskip

The inspiration for this work, which is the subject of the author's PhD thesis \cite{PhD}, is the following result of Diaconis and Shahshahani \cite{DS}:

\medskip

\noindent {\bf Theorem.}  (Upper Bound Lemma).
\emph{Let $\nu$ be a probability on a finite group $G$ and $k\in \mathbb{N}$. Then}
\begin{align*}
\|\nu^{\star k}-\pi\|^2\leq \frac14 \sum_{\alpha\in\operatorname{Irr}(G)\backslash \tau}d_\alpha \operatorname{Tr}(\left(\widehat{\nu}(\alpha)^\ast\right)^k\widehat{\nu}(\alpha)^k)\qquad\bullet
\end{align*}

The sum is over non-trivial irreducible representations and $\widehat{\nu}(\alpha)$ is the \emph{Fourier transform of $\nu$ at the representation $\rho^\alpha$}
\begin{align*}
\widehat{\nu}(\alpha)=\sum_{t\in G}\nu(t)\rho^\alpha(t).
\end{align*}

\bigskip

As an application, Diaconis and Shahshahani show that  $k_n=\frac12\,n\ln n$ random transpositions is necessary and sufficient to `mix up' a deck of $n$ cards. This result comprises two elements: an upper bound at $k_n+c\,n$ derived using the Upper Bound Lemma, as well as a lower bound at $k_n-c\,n$ derived using a carefully selected test set. A phenomenon occurs in this example that as $n\rightarrow \infty$, the lower bound for $k_n-c\,n$ tends to one for large $c$, and the upper bound for $k_n+c\,n$ approaches zero for large $c$. Note that $\frac{2cn}{k_n}\rightarrow 0$ so that as $n\rightarrow \infty$, the convergence from far from random to close to random is abrupt. This behaviour, called the \emph{cut-off phenomenon}, occurs in a number of examples of random walks on finite groups. Saloff-Coste \cite{Saloff} is an excellent survey that includes more ideas even than the excellent monograph of Diaconis \cite{DiaconisBook}. The author's MSc thesis \cite{MSc} also comprises a leisurely introduction to the subject.

\bigskip

Finite quantum groups are a noncommutative or `quantum' generalisation of finite groups, and the central result of this work is that the Diaconis--Shahshahani Upper Bound Lemma also holds for finite quantum groups.

\bigskip

Considered as a research programme, quantum probability is concerned with generalising, where possible, objects in the study of classical probability to quantum objects in the study of quantum probability theory. It is under this programme that this work lies: the study of random walks on finite groups uses classical probability theory --- a study of random walks on \emph{quantum} groups \emph{should} be the corresponding area of study in quantum probability.

\bigskip

Therefore, this work is concerned with a generalisation of a generalisation of card shuffling: generalising, where possible, the ideas and results of Diaconis and Shahshahani to the case of finite quantum groups. While the central object of card shuffling --- the set of shuffles --- is generalised to that of a set of elements of a group, the generalisation to \emph{quantum} groups moves away from a `set of points' interpretation.

\bigskip

However, in the study of random walks on finite quantum groups at least, the quantum theory generalises so nicely from the classical setup that it can be fruitful to refer to the \emph{virtual} object that is a quantum group as if it really exists. At the very least this approach gives a most pleasing notation for quantum objects. At its very best it can inspire the noncommutative geometer to make good choices in terms of definitions, etc.

\bigskip

One of the most exciting and potentially lucrative aspects of quantum probability, or rather more specifically quantum group theory, is that theorems about finite groups may in fact be true for quantum groups also. For example, the classical finite Peter--Weyl Theorem  concerning the matrix elements of representations of classical groups is exactly the same as the finite Peter--Weyl Theorem \ref{QPW} for finite \emph{quantum} groups in the sense that replacing in the classical statement `finite group' with `finite quantum group' yields the quantum statement. What this really means is that the classical finite Peter--Weyl Theorem is actually just a special case of the quantum finite Peter--Weyl Theorem.

\bigskip

This is rather comforting on the conceptional level --- these quantum objects behave so much like their `set of points' classical counterparts --- but in this study it is on the pragmatic level of proving results about these quantum objects that this principle really comes to the fore. On the one hand, some theorems concerning the theory of finite groups are just  corollaries to results about quantum groups. On this other, pragmatic, hand, there is a translation principle: any proof of a classical group theorem, written without regard to any of the points in the `set of points', may be directly translatable into a proof of the corresponding quantum group theorem.

\bigskip

For example, the Haar state, $h$, can allow  classical `sum over points' arguments and statements about the algebra of  functions on a finite classical group to be translated:
$$\underbrace{\frac{1}{|G|}\sum_{t\in G}f(t)}_{\text{classical: references points }t\in G}=\underbrace{h(f)}_{\text{quantum: no reference to points}}.$$
Representation Theory and `sum over points' arguments, therefore, are translatable and it is precisely these ideas that play a central role in the representation-theoretic approach of Diaconis and Shahshahani to analysing the rate of convergence of random walks on finite groups \cite{DS}. Once the quantum versions of the various objects and maps used by Diaconis and Shahshahani are established, it is largely straightforward to derive and prove the quantum Diaconis--Shahshahani Upper Bound Lemma.

\bigskip

The restriction to \emph{finite} quantum groups is for two reasons. Firstly the classical work that this work is building upon is the Diaconis--Shahshahani Theory approach to random walks on \emph{finite}  groups. Secondly, the approach to finding the quantum versions of classical objects in this work --- extolled in Section \ref{FQG} --- requires the isomorphism $\mathbb{C}(X\times Y)\cong \mathbb{C}X\otimes\mathbb{C}Y$ for sets $X$ and $Y$, and this holds only when $X$ and $Y$ are finite.

\bigskip

However, the classical Diaconis--Shahshahani Theory has been extended to possibly-infinite compact groups by Rosenthal \cite{Ros}. Freslon \cite{Amaury,Amaury2} has also extended this work on finite quantum groups, developed during the author's PhD studies, to the far more technically-involved  case of compact quantum groups, where the algebras are no longer necessarily finite dimensional. Freslon presents many interesting examples of random walks on compact quantum groups --- with the cut-off phenomenon --- and also addresses, particularly in \cite{Amaury}, a number of concerns not answered or addressed in the author's PhD thesis. As there is a generalisation of the classical finite symmetric group that is infinite dimensional for $n\geq 4$, it would be remiss not to point out that the restriction to \emph{finite} quantum groups is more than a little unnatural, and also brushes many technical difficulties under the carpet.

\bigskip

Therefore this work should be considered a first, pioneering, attempt at a quantum generalisation of the work of Diaconis and Shahshahani, with Freslon's papers a far more accomplished and reaching generalisation.

\bigskip

The paper is organised as follows. In Section 1, category theory language is used to motivate the definition of a finite quantum group. Results concerning the dual of a finite quantum group that are used in the sequel are presented at this point. Finally the definition of a random walk on a finite quantum group is given. In Section 2, the total variation distance, used for measuring the `distance to random' of a random walk, is introduced. The central result is contained in Section 3; this leans so heavily on the work of Van Daele, that the author calls it by Diaconis--Van Daele Theory. In Section 4, two applications of the Upper Bound Lemma are presented: to a random walk on the dual of $S_n$; and to a random walk on the truly quantum groups of Sekine. In neither example is the cut-off phenomenon exhibited, and in the case of the random walk on the quantum groups of Sekine it is shown the cut-off phenomenon does \emph{not} occur.

\section{Preliminaries}

\subsection{Finite Quantum Groups\label{FQG}}
The following approach to introducing finite quantum groups is a brief summary of the approach outlined in detail in Section 2.2 of \cite{PhD}.

\bigskip

A finite classical group $G$ is an object in the category of finite sets, and the group multiplication $G\times G\rightarrow G$,  inclusion of the unit, and inverse $G\rightarrow G$, are morphisms in this category. The three group axioms are given by three commutative diagrams.
%\newpage

 \bigskip

 Let  $\mathbb{C}G$ be the complex vector space with basis elements $\{\delta^s:s\in G\}$. The map $G\mapsto \mathbb{C}G$ is a covariant functor, the $\mathbb{C}$ functor, from the category of finite sets to the category of finite dimensional complex vector spaces. The image of the group multiplication is a multiplication, $\nabla: \mathbb{C}G\otimes \mathbb{C}G\rightarrow \mathbb{C}G$ given by $\nabla(\delta^s\otimes\delta^t)=\delta^{st}$. The presence of the tensor product comes via the isomorphism $\mathbb{C}(G\times G)\cong \mathbb{C}G\otimes \mathbb{C}G$ which does not hold if $G$ is infinite. This gives $\mathbb{C}G$ the structure of a complex associative algebra called the \emph{group algebra of }$G$.  The image of the three commutative diagrams (which comprise the group axioms) under the functor $\mathbb{C}$ give relations that hold in $\mathbb{C}G$, and if $\displaystyle\varphi=\sum_{t\in G}\alpha_t\delta^t$, then the involution
$$\varphi^*=\sum_{t\in G}\overline{\alpha_t}\delta^{t^{-1}}$$
gives the group algebra the structure of a *-algebra.

\bigskip

Now apply the contravariant dual endofunctor to the group algebra. The dual endofunctor maps a finite dimensional complex vector space $U$ to its dual $U^*$, and a morphism $T:U\rightarrow V$ to a morphism $V^*\rightarrow U^*$ given by $f\mapsto f\circ T$. The dual of the group algebra is the algebra of functions on $G$, $F(G)$, with dual basis $\{\delta_s:s\in G\}$ such that $\delta_s(\delta^t)=\delta_{s,t}$, and involution $f^*(\delta^s)=\overline{f(\delta^s)}$.  The composition of the $\mathbb{C}$ functor and the dual endofunctor gives a functor $G\mapsto F(G)$ and the images of the group multiplication, inverse, and inclusion of the unit, under this functor composition give structure maps on $F(G)$. These structure maps are called the \emph{comultiplication}, an algebra homomorphism $\Delta:F(G)\rightarrow F(G)\otimes F(G)$, $\delta_s\mapsto \mathds{1}_{m^{-1}(s)}\equiv \sum_{t\in G}\delta_{st^{-1}}\otimes \delta_t$ (for finite $G$); the counit, $\varepsilon: F(G)\rightarrow \mathbb{C}$, $\delta_{s}\mapsto \delta_{s,e}$; and the antipode, $S:F(G)\rightarrow F(G)$, $\delta_s\mapsto \delta_{s^{-1}}$. The associativity, inverse, and identity group axiom commutative diagrams, under this functor, give \emph{coassociativity}, the \emph{counital property}, and the \emph{antipodal property}:

\begin{align}
(\Delta\otimes I_{F(G)})\circ \Delta&=(I_{F(G)}\otimes \Delta )\circ \Delta\nonumber
\\ (\varepsilon\otimes I_{F(G)})\circ \Delta&=I_{F(G)}= (I_{F(G)}\otimes \varepsilon)\circ \Delta \label{Hopf}
\\ M\circ (S\otimes I_{F(G)})\circ \Delta&=\eta_{F(G)}\circ \varepsilon=M\circ (I_{F(G)}\otimes S)\circ \Delta\nonumber
\end{align}

\bigskip

Here $M:F(G)\otimes F(G)\rightarrow F(G)$ is pointwise multiplication and $\eta_{F(G)}$ is the inclusion of the unit, $\mathds{1}_G$. With the fact that $f^*f=0$ if and only if $f=0$, $F(G)$ can be given the structure of a finite dimensional $\mathrm{C}^*$-algebra. It is a unital commutative algebra and the group axioms are encoded by these relations. Note furthermore that $\Delta$ satisfies $\Delta(f^*)=\Delta(f)^*$, where the involution in $F(G)\otimes F(G)$ is given by $(f\otimes g)^*=f^*\otimes g^*$.

%\bigskip
\newpage
A finite dimensional $\mathrm{C}^*$-algebra $A$ with a *-homomorphism $\Delta:A\rightarrow A\otimes A$, and maps $\varepsilon:A\rightarrow \mathbb{C}$, and $S:A\rightarrow A$, that satisfies the above relations, but is not necessarily commutative is thus considered the \emph{algebra of functions on a finite quantum group}. Such an algebra is called a $\mathrm{C}^*$-Hopf algebra.

\begin{definition}
An \emph{algebra of functions on a finite quantum group}, is a unital $\mathrm{C}^*$-Hopf algebra $A$; that is a $\mathrm{C}^*$-algebra $A$ with a *-homomorphism $\Delta$, a counit $\varepsilon$, and an antipode $S$, satisfying the relations (\ref{Hopf}).
\end{definition}
Denote an algebra of functions on a finite quantum group by $A=:F(G)$ and refer to $G$ as a finite quantum group, with unit  denoted by $\mathds{1}_{G}$, and note that $S$ is an involution (not necessarily true for infinite $G$, see remark to Theorem 2.2, \cite{VD2}). Timmermann presents in Chapter 1 further properties of Hopf algebras, for example the fact that the counit is a *-homomorphism (used in the sequel). The simplest noncommutative example of an algebra of functions on a finite quantum group is $\mathbb{C}S_3$, where $S_3$ is the classical symmetric group on three elements, where the comultiplication is given by $\Delta_{\mathbb{C}S_3}(\delta^\sigma)=\delta^\sigma\otimes \delta^\sigma$, and is the dual of the pointwise multiplication in $F(S_3)$. Where $\tau$ is the flip map $a\otimes b\mapsto b\otimes a$, this comultiplication has the property that $\tau\circ \Delta_{\mathbb{C}S_3}=\Delta_{\mathbb{C}S_3}$. Algebras of functions on finite quantum groups, $F(C)$, whose comultiplications have this property, $\tau\circ \Delta=\Delta$ are said to be \emph{cocommutative}, and are of the form $F(C)=\mathbb{C}G$ for $G$ a classical finite group. Similarly a commutative algebras of functions on a finite quantum group is the algebra of functions on some classical finite group (can be inferred from Theorem 3.3, \cite{VK}).

\bigskip

For the remainder of this work, unless explicitly stated otherwise, $G$ is a quantum group, and assumed finite. There are well established notions of compact and locally compact quantum groups. See, for example, Timmermann \cite{Timm} for more.

\subsection{The Dual of a Quantum Group}
Where $\pi$ is the random/uniform distribution on a classical group $G$, consider the functional on $F(G)$:

$$f\mapsto \int_G f(t)\,d\pi(t)=\frac{1}{|G|}\sum_{t\in G}f(t).$$

This is a normalised, positive functional, invariant in the sense that for each $s\in G$,

\begin{align}
 \int_G f(t)\,d\pi(t)&=\int_G f(st)\,d\pi(t)=\int_G f(ts)\,d\pi(t)\nonumber
 \\ \Rightarrow \mathds{1}_G\cdot \int_G f&=\left(I_{F(G)}\otimes \int_G\right)\Delta(f)= \left(\int_G\otimes I_{F(G)}\right)\Delta(f)\,.\,\,\qquad (f\in F(G))\label{invar}
\end{align}

This normalised, invariant, positive functional is called a \emph{Haar state}, it is unique, and is denoted by $\int_G$ (earlier $h$ was used).
\newpage
 A finite \emph{quantum} group also has a unique Haar state (Theorem 1.3, \cite{VD2}), denoted by $\int_G$. Invariance of the Haar state is given by (\ref{invar}) and furthermore it is  tracial: $\int_Gab=\int_Gba$ (Theorem 2.2, \cite{VD2}).

\bigskip

Consider the space, $\widehat{A}$, of linear functionals on $A:=F(G)$ of the form
 $$b\mapsto \int_{G} ba\,.\,\,\qquad (a,\,b\in F(G))$$
As $F(G)$ is finite dimensional, the continuous and algebraic duals coincide. Furthermore,  the Haar state is faithful and so
 $$\langle a,b\rangle:=\int_{G} a^*b$$
  defines an inner product making $F(G)$ a Hilbert space. Via the Riesz Representation Theorem for Hilbert spaces, for every element  $\varphi\in F(G)'$, there exists a density $a_{\varphi}^*\in F(G)$ such that:
  $$\varphi(b)=\langle a_{\varphi}^*,b\rangle=\int_G a_{\varphi}b\,,\,\,\qquad (b\in F(G))$$
  so that $F(G)'=\widehat{A}$. This space can be given the structure of an algebra of functions on a finite quantum group by employing the contravariant dual functor to $F(G)$ and its structure maps. The dual functor maps the object $F(G)$ to $\widehat{A}$; the comultiplication $\Delta:F(G)\rightarrow F(G)\otimes F(G)$ to an associative multiplication $\nabla:\widehat{A}\otimes \widehat{A}\rightarrow \widehat{A}$; and the multiplication $M:F(G)\otimes F(G)\rightarrow F(G)$ to a comultiplication $\widehat{\Delta}: \widehat{A}\rightarrow \widehat{A}\otimes \widehat{A}$ (see \cite{PhD}, Section 2.5 for more). The *-involution on $\widehat{A}$ is given by:

    $$\varphi^*(a)=\overline{\varphi(S(a)^*)}\,.\,\,\qquad (\varphi\in\widehat{A},\,a\in F(G))$$

    The finite quantum group formed in this way is called the \emph{dual} of the quantum group $G$, and is denoted by $\widehat{G}$. An invariant, positive functional on $F(\widehat{G})$ is given by $\widehat{h}(\varphi)=\varepsilon(a_{\varphi})$. In the case of a classical group $G$, $F(\widehat{G})=\mathbb{C}G$.

  \bigskip

   The analogue of a probability distribution on a classical group is a state on the algebra of functions on a quantum group, that is a positive functional of norm one. Denote the states on $F(G)$ by $M_p(G)$ and note that $M_p(G)$ is a subset of $F(\widehat{G})$. A density $a_\nu$ defines a state $\nu\in M_p(G)$ if and only if $a_\nu$ is positive and $\int_G a_\nu=1$. There are Convolution and Plancherel Theorems for $F(\widehat{G})$ that will be used to prove the Upper Bound Lemma.  The Convolution Theorem is well presented by Bhowmick, Skalski, and So{\l}tan (Section 1.1, \cite{Fourier}), which heavily cites Van Daele \cite{VD1,VD2,VD3,VD4}; while the Plancherel Theorem is well presented by Timmermann  (Section 2.3, \cite{Timm}), which also cites Van Daele \cite{VD2}.

   \bigskip

The \emph{convolution product} on $F(G)$ is given by:
 \begin{align*}
 a\ostar  b:&=\left(\int_G\otimes I_{F(G)}\right)\left(\left((S\otimes I_{F(G)})\Delta(b)\right)(a\otimes\mathds{1}_G)\right).
% \\&= \sum b_{(2)}\int_{G}\left( S(b_{(1)})a\right)\label{conv2}
 \end{align*}
 There is a Convolution Theorem relating this convolution to a convolution in $F(\widehat{G})$:
$$\varphi_1\star\varphi_2:=(\varphi_1\otimes\varphi_2)\Delta\,.\,\,\qquad (\varphi_1,\,\varphi_2\in F(\widehat{G}))$$

\bigskip

Within the proof of Proposition 3.11 of \cite{VD3}, Van Daele proves the following identity:
   \begin{align*}
   S\left(\left(I_{F(G)}\otimes \int_G\right)\left(\Delta(a)(\mathds{1}_G\otimes b)\right)\right)=\left(I_{F(G})\otimes\int_G\right)\left((\mathds{1}_G\otimes a)\Delta(b)\right)\,.\,\,\qquad (a,\,b\in F(G))
   \end{align*}
This identity can be used to prove the following (see Section 1.1 of \cite{Fourier}):

 \begin{theorem}
(Van Daele's Convolution Theorem)\label{VDCT}
 For  $\varphi_1,\varphi_2\in F(\widehat{G})$
 $$a_{\varphi_1\star\varphi_2}=a_{\varphi_1}\ostar a_{\varphi_2}\qquad\bullet$$
 \end{theorem}
Timmerman (Lemma 2.3.8, \cite{Timm}) presents a proof of the following lemma.
\begin{lemma}
Let $\varphi_1,\,\varphi_2\in F(\widehat{G})$. Then
\begin{align*} \widehat{h}(\varphi_1\varphi_2)=\varphi_1(S^{-1}(a_{\varphi_2}))\qquad\bullet\end{align*}
\end{lemma}

\bigskip

%\newpage
\begin{theorem}
(Plancherel Theorem)\label{Planch}
Let $G$ be a quantum group and $\varphi \in F(\widehat{G})$. Then
$$\widehat{h}\left(\varphi^*\varphi\right)=\int_{G}a_{\varphi}^*a_{\varphi}.$$
\end{theorem}
\begin{proof}
Applying the previous lemma to the left-hand side:
\begin{align*}
\widehat{h}(\varphi^*\varphi)&=\varphi^*\left(S^{-1}(a_{\varphi})\right)= \overline{\varphi\left(S(S^{-1}(a_{\varphi}))^*\right)}
\\&=\overline{\varphi(a_{\varphi}^*)}=\overline{\int_G a_{\varphi}^*a_{\varphi}}=\int a_{\varphi}^*a_{\varphi},
\end{align*}
where the traciality of $\int_G$ was used $\bullet$
\end{proof}

\subsection{Random Walks on Finite Quantum Groups}
 Early work on quantum stochastic processes by various authors led to random walks on duals of compact groups (particularly Biane, see \cite{FG} for references), and other examples, but Franz and Gohm \cite{FG} defined with clarity
a \emph{random walk on a finite quantum group}. Given a random walk on a finite classical group $G$, by applying the functor composition referenced in Section \ref{FQG} to the random variables $\xi_k:G^{k}\rightarrow G$, random variables on the level of $F(G)$, $j_k:F(G)\rightarrow \bigotimes_{k\text{ copies}}F(G)$ can be defined by:
$$j_k(f)=f\circ \delta^{\xi_k},$$
and the generalisation to random walks on finite quantum groups can be made in this way. It is also possible to generalise from classical to quantum random walks using stochastic operators, and also $F(G)$ analogues of the random variables $s_i$ (denoted by $k_i$ by Franz and Gohm and $z_i$ by the author in \cite{PhD}).

\bigskip

 In Section 3.2 of \cite{PhD}, this generalisation of Franz and Gohm from random walks on finite classical groups to random walks on  finite quantum groups is explored in detail, and concludes with Franz and Gohm's assertion that a random walk on a finite quantum group $G$ driven by a state $\nu\in M_p(G)$ is essentially the same thing as the semigroup of convolution powers, $\{\nu^{\star k}\}$, defined inductively through
$$\nu^{\star (k+1)}=(\nu\otimes \nu^{\star k})\circ \Delta.$$
Indeed, Amaury Freslon defines a random walk on a (compact) quantum group simply as a state on $C(G)$, and proceeds to study the semigroup of convolution powers without mentioning the random variables $j_k$ nor $z_i$, only mentioning the stochastic operator \newline $P_\nu=(\nu\otimes I_{F(G)})\circ \Delta$ for analytical reasons (see Section 4.2, \cite{Amaury}).
\section{Total Variation Distance\label{dist}}
In the classical case, the distance between $\nu$ and $\mu \in M_p(G)$ is given by the \emph{total variation distance}:
$$\|\nu-\mu\|:=\sup_{S\subset G}|\nu(\mathds{1}_S)-\mu(\mathds{1}_S)|.$$
Given a subset $S\subset G$, the indicator function on $S$, $\mathds{1}_{S}=\sum_{s\in S}\delta_s$, is a projection, and all projections in $F(G)$ are formed in this way, and so there is a one-to-one correspondence between projections in $F(G)$ and subsets of $G$.
Using this correspondence, this definition can be extended to quantum groups:
$$\|\nu-\mu\|:=\sup_{p\in F(G)\text{ a projection}}|\nu(p)-\mu(p)|.$$

\bigskip

The Haar state is a normal, faithful trace, therefore non-commutative $\mathcal{L}^p$ machinery \cite{LP} can be used to put $p$-norms on $F(G)$:
\begin{align*}
\|a\|_p:=\left(\int_{G}|a|^p\right)^{1/p}\,.\,\,\qquad (a\in F(G))
\end{align*}
Set the infinity norm equal to the operator norm.

\begin{proposition}\label{tvd}
Let $G$ be a quantum group and $\nu,\,\mu\in M_p(G)$:
$$\|\nu-\mu\|=\frac12 \left\|a_\nu-a_\mu\right\|_1.$$
\end{proposition}
\begin{proof}
Standard non-commutative $\mathcal{L}^p$ machinery from \cite{LP} shows that:
\begin{align*}
\left\|a_\nu-a_\mu\right\|_1=\sup_{\phi\in F(G):\|\phi\|_\infty\leq 1}\left|\int_G ((a_\nu-a_\mu)\phi)\right|=\sup_{\|\phi\|_{\infty}\leq 1}\left|\nu(\phi)-\mu(\phi)\right|.
\end{align*}

\bigskip

Embed the algebra of functions $F(G)$ into a matrix algebra. As everything is in finite dimensions, the supremum on the right is actually a maximum. Suppose that the two states $\nu$ and $\mu$ are given by density matrices: $\nu=\operatorname{Tr}(\rho_{\nu}\cdot)$ and $\mu=\operatorname{Tr}(\rho_{\mu}\cdot)$, and $\theta:=\rho_{\nu}-\rho_{\mu}$ so that
$$\|a_\nu-a_\mu\|_1=\max_{\|\phi\|_{\infty}\leq 1}|\operatorname{Tr}((\rho_{\nu}-\rho_{\mu})\phi)|=\max_{\|\phi\|_{\infty}\leq 1}|\operatorname{Tr}(\theta\phi)|.$$
Using the $M_n(\mathbb{C})$ inequality ($\|A\|_1= \operatorname{Tr}(|A|)$, $|A|:=\sqrt{A^*A}$), $|\operatorname{Tr}(AB)|\leq \|A\|_1\|B\|_\infty,$
\begin{align*}
  \|a_\nu-a_\mu\|_1 & = \max_{\|\phi\|_\infty\leq 1}|\operatorname{Tr}(\theta \phi)| \\
   & \leq  \max_{\|\phi\|_\infty\leq 1}\|\theta\|_1\|\phi\|_\infty=\|\theta\|_1.
\end{align*}
Note that $\theta$ is self-adjoint so that there exists a polar decomposition $\theta=U|\theta|$ such that $U$ is self-adjoint. Consider $\phi=U$, the phase of $\theta$, so that $U\theta=|\theta|$, and so
$$\operatorname{Tr}(\theta U)=\operatorname{Tr}(|\theta|)=\|\theta\|_1,$$
that is equality is achieved at the self-adjoint unitary $U$:
$$\|a_\mu-a_\nu\|_1=|\nu(U)-\mu(U)|.$$

\bigskip

Consider $q=\frac{\mathds{1}_G+U}{2}$. The self-adjointness of $U$ gives $q=q^*$, and the fact $U$ is also a unitary implies that $U$ is an involution. This implies that $q=q^2$ and so $q$ is a projection. The projection distance satisfies
\begin{align*}
\|\nu-\mu\|\geq |\nu(q)-\mu(q)|=\frac{1}{2}|\nu(U)-\mu(U)|=\frac12 \|a_\nu-a_\mu\|_1.
\end{align*}

\bigskip

Now suppose the projection distance is attained at $p$ and consider the infinity-norm-one element $\phi=2p-\mathds{1}_G \in F(G)$:
\begin{align*}
\|a_\nu-a_\mu\|_1&\geq |\nu(2p-\mathds{1}_G)-\mu(2p-\mathds{1}_G)|=2|\nu(p)-\mu(p)|=2\|\mu-\nu\|
\\ \Rightarrow \|\nu-\mu\|&\leq \frac12 \|a_\nu-a_\mu\|_1\qquad \bullet
\end{align*}
\end{proof}

\bigskip

This proof takes advantage of the finite dimensionality of $F(G)$. Amaury Freslon (\cite{Amaury}, Lemma 2.6) states and proves a similar result for $\left\|\nu-\int_G\right\|$ which holds for not-necessarily finite-dimensional compact quantum groups whenever $\nu$ is a state on $\mathcal{L}^{\infty}(G)$ that has an $\mathcal{L}^1(G)$-density $a_\nu$, which is not always the case in infinite dimensions.

\bigskip

The total variation distance defined using projections has three properties that suggest it is a suitable generalisation of the classical variation distance.

\bigskip

Firstly, in the case of a classical group, $G$, the one-to-one correspondence between projections and subsets, and between probability measures on $G$ and states on $F(G)$, means that there is agreement in the classical case. Furthermore, there is a one-to-one correspondence between representations of a classical group $G$ and \emph{co}representations of $F(G)$ and this implies that estimates of the distance to random are the same whether the classical or quantum Upper Bound Lemmas are used.

\bigskip

The Upper Bound Lemma is a formula for a two norm, and thus requires a Cauchy--Schwarz inequality to relate back to the total variation distance. Such an inequality for total variation distance is provided for by non-commutative $\mathcal{L}^p$ machinery (see \cite{LP}):
$$\|ab\|_1\leq \|a\|_2\|b\|_2\,,\,\,\qquad (a,\,b\in F(G))$$
applied to
\begin{align*}
\|\nu-\mu\|^2&=\frac14 \|a_\nu-a_\mu\|_1^2=\frac{1}{4}\|(a_\nu-a_\mu)\mathds{1}_G\|_1^2
\\ &\leq \frac14 \|a_\nu-a_\mu\|_2^2\|\mathds{1}_G\|_2^2=\frac14 \,\widehat{h}\left((\nu-\mu)^*\star(\nu-\mu)\right)
\end{align*}
The final equality here is the Plancherel Theorem \ref{Planch}.

\bigskip

Finally, to establish lower bounds, the total variation distance has a presentation(s) as a supremum:

\begin{equation}\|\nu-\mu\|=\sup_{p\in F(G)\text{ a projection}}|\nu(p)-\mu(p)|=\frac12\,\sup_{\|\phi\|_{\infty}\leq 1}\left|\nu(\phi)-\mu(\phi)\right|\label{lb}\end{equation}

\section{Diaconis--Van Daele Theory}
\subsection{Corepresentation Theory}
\begin{definition}
A \emph{corepresentation} of the algebra of functions on a quantum group $G$ on a complex vector space $V$ is a linear map $\kappa:V\rightarrow V\otimes F(G)$ that satisfies:
$$(\kappa\otimes I_{F(G)})\circ \kappa=(I_V\otimes\Delta)\circ\kappa\qquad\text{and}\qquad(I_V\otimes \varepsilon)\circ \kappa=I_V.$$
\end{definition}
If $V$ is equipped with a Hermitian inner product, $\langle,\rangle_V$, define for $v,w\in V$, and \newline $a,b\in F(G)$
$$\langle v\otimes a,w\otimes b\rangle_{F(G)}=\langle v,w\rangle_V\cdot a^*b.$$
A representation $\kappa$ on such a vector space is said to be \emph{unitary} if $\langle \kappa(v),\kappa(w)\rangle_{F(G)}=\langle v,w\rangle_V\cdot \mathds{1}_G$.

\bigskip

Let $\kappa$ be a corepresentation on a vector space $V$. Denote by $\overline{V}$ the conjugate vector space of $V$ and by $v\mapsto \overline{v}$ the canonical conjugate-linear isomorphism. Since $\Delta$ and $\varepsilon$ are $*$-homomorphisms, the map
$$\overline{\kappa}:\overline{V}\rightarrow \overline{V}\otimes F(G)\,,\,\,\,\overline{e_j}\mapsto \sum \overline{e_i}\otimes \rho_{ij}^*,$$
is a representation again, called the \emph{conjugate} of $\kappa$.

\bigskip

Denoting the map $a_\varphi\mapsto \varphi$ by $\mathcal{F}$, so that  $\mathcal{F}(a_\varphi)=\varphi$, for each $a_\varphi\in F(G)$ define the map $\widehat{a_\varphi}\in L(\overline{V})$:
$$\widehat{a_\varphi}(\kappa)=(I_{\overline{V}}\otimes \mathcal{F}(a_\varphi))\circ \overline{\kappa}=(I_{\overline{V}}\otimes \varphi)\circ \overline{\kappa}.$$
\begin{proposition}
\label{conv1}
For $a_{\varphi_1},\,a_{\varphi_2}\in F(G)$ and $\kappa$ a representation of $G$ on $V$
$$\widehat{a_{\varphi_1\star \varphi_2}}(\kappa)=\widehat{a_{\varphi_1}}(\kappa)\circ \widehat{a_{\varphi_2}}(\kappa).$$
Furthermore, if $\kappa$ is unitary then
\begin{equation}\widehat{a_{\varphi^*}}(\kappa)=\widehat{a_\varphi}(\kappa)^*.\label{star}\end{equation}
Note the first involution is in $F(\widehat{G})$ and the second is in $L(\overline{V})$.
\end{proposition}
\begin{proof}
The proposition is that the map $F(\widehat{G})\rightarrow L(\overline{V})$, $\varphi\mapsto \widehat{a_\varphi}(\kappa)$ is a $*$-homomorphism. See Timmermann for a proof (Proposition 3.1.7 ii. and v., \cite{Timm}) $\bullet$

\end{proof}
The following result is presented by Timmermann (Proposition 3.29 and Theorem 3.2.12, \cite{Timm}), but is due to Woronowicz \cite{Pseudogroups}
\newpage

\begin{theorem}
\emph{(Finite Quantum Peter--Weyl Theorem)}\label{QPW}
Let $\mathcal{I}=\operatorname{Irr}(G)$ be an index set for a family of pairwise-inequivalent irreducible unitary representations of $G$. If $d_{\alpha}$ is the dimension of the vector space on which $\rho^\alpha$ acts ($\alpha\in\mathcal{I}$),
$$\left\{\rho_{ij}^\alpha\,|\,i,j=1\dots d_\alpha,\,\alpha\in \mathcal{I}\right\},$$
the set of matrix elements of $\mathcal{I}$, is an orthogonal basis of $F(G)$ with respect to the inner product
\begin{align*}
\langle a,b\rangle:=\int_{G}a^*b\,,\,\,\qquad (a,\,b\in F(G))
\end{align*}
such that $\langle \rho_{ij}^\alpha,\rho_{ij}^\alpha \rangle =1/d_{\alpha}\qquad\bullet$
\end{theorem}
\subsection{Fourier Transform}
The following definition is similar to that of Simeng Wang ((2.5), \cite{Simeng}) save for a choice of left-right. As remarked upon by Simeng Wang, his definition is similar to earlier definitions of Kahng and also Caspers save for the presence of the conjugate representation $\overline{\kappa}_\alpha$ rather than $\kappa_{\alpha}$ itself. As Wang explains, the conjugate representation is used to be compatible with standard definitions in classical analysis on compact groups and hence most welcome for this work.
\begin{definition}
(The Fourier Transform)
Let $G$ be a quantum group with representation notation as before. Then the \emph{Fourier transform} is a map:
$$F(G)\rightarrow \bigoplus_{\alpha\in\operatorname{Irr}(G)}L(\overline{V}_{\alpha})\,,\,\,a_\varphi\mapsto \widehat{a_\varphi},$$
defined for each $\alpha\in \operatorname{Irr}(G)$:
$$\widehat{a_\varphi}(\alpha)=(I_{\overline{V}}\otimes \mathcal{F}(a_\varphi))\circ \overline{\kappa_\alpha}=\left(I_{\overline{V_\alpha}}\otimes \varphi\right)\circ \overline{\kappa_\alpha}.$$
Each $\widehat{a_\varphi}(\alpha)$ is called the \emph{Fourier transform of $a_\varphi$ at the representation $\kappa_\alpha$}.
For $\varphi\in F(\widehat{G})$:
$$\widehat{\varphi}(\alpha):=\widehat{a_\varphi}(\alpha)=(I_{\overline{V_\alpha}}\otimes \varphi)\circ\overline{\kappa_\alpha},$$
so $\widehat{\varphi}$ is identified with $\widehat{a_\varphi}$.
\end{definition}
The maps $\{\widehat{a_\varphi}(\alpha)\,:\,\alpha\in\operatorname{Irr}(G)\}$ play a key role in the sequel.

\bigskip
%\newpage
\begin{theorem}
(Diaconis--Van Daele Inversion Theorem)\label{DVDIT}
Let $\varepsilon$ be the counit of a quantum group $G$ and $\varphi\in F(\widehat{G})$. Then
\begin{align*}\widehat{h}(\varphi):=\varepsilon\left(a_\varphi\right)=\sum_{\alpha\in\text{Irr}(G)} d_\alpha \operatorname{Tr}\left(\widehat{a_{\varphi}}\left(\alpha\right)\right).\end{align*}
where the sum is over the irreducible unitary representations of $F(G)$.
\end{theorem}
\begin{proof}
Both sides are linear in $a_\varphi$ so it suffices to check $a_{\varphi}=\rho_{kl}^{\beta}$ for $\beta\in\text{Irr}(G)$. The left-hand side reads
$$\varepsilon\left(\rho_{kl}^\beta\right)=\delta_{k,l}.$$
To calculate the right-hand-side, calculate for a given representation the trace of $\widehat{\rho_{kl}^\beta}\left(\alpha\right)$. Let $\overline{e_j}\in \overline{V_\alpha}$ and calculate
\begin{align*}
\widehat{\rho_{kl}^{\beta}}\left(\alpha\right)\overline{e_j}&=\left(I_{\overline{V_\alpha}}\otimes\mathcal{F}\left(\rho_{kl}^\beta\right)\right)\sum_{i} \overline{e_i}\otimes\left(\rho_{ij}^\alpha\right)^*
\\&=\sum_i \mathcal{F}\left(\rho_{kl}^\beta\right)\left(\rho_{ij}^{\alpha}\right)^* \overline{e_i}=\sum_i \int_{G}\left(\left(\rho_{ij}^{\alpha}\right)^*\rho_{kl}^\beta\right) \overline{e_i}
\end{align*}
This is zero unless $\alpha\equiv \beta$. If $\alpha\equiv\beta$ then
\begin{align*}
\widehat{\rho_{kl}^{\beta}}\left(\alpha\right)\overline{e_j}&=\int_{G}\left(\left(\rho_{kj}^\beta\right)^*\rho_{kl}^\beta\right)\cdot\overline{e_k}=\frac{1}{d_\beta}\delta_{j,l}\overline{e_k}
\\ \Rightarrow \operatorname{Tr}\left(\widehat{\rho_{kl}^\beta}(\beta)\right)&=\sum_j\langle \overline{e_j},\widehat{\rho_{kl}^\beta}(\beta)\overline{e_j}\rangle_{\overline{V_\beta}}=\sum_j \left\langle \overline{e_j},\frac{1}{d_\beta}\delta_{j,l}\overline{e_k}\right\rangle=\frac{1}{\delta_\beta}\delta_{k,l}.
\end{align*}

Multiply this by $d_\beta$ to get $\delta_{k,l}$ $\bullet$
 \end{proof}

 \bigskip

 \begin{theorem}
(Diaconis--Van Daele Convolution Theorem)\label{DVDCT}
For a representation $\kappa_\alpha$ of $G$ and $a_{\varphi_1},\,a_{\varphi_2}\in F(G)$
 $$\widehat{a_{\varphi_1}\ostar a_{\varphi_2}}\left(\alpha\right)=\widehat{a_{\varphi_1}}\left(\alpha\right)\circ\widehat{a_{\varphi_2}}\left(\alpha\right).$$
\end{theorem}
 \begin{proof} Recalling the identifications $\widehat{a_\varphi}(\alpha)=:\widehat{\mathcal{F}(a_\varphi)}(\alpha)=\widehat{\varphi}(\alpha)$, and that
 \begin{align*}
 \widehat{\left(a_{\varphi_1}\ostar  a_{\varphi_2}\right)}\left(\alpha\right)&=\left(I_{\overline{V_\alpha}}\otimes\mathcal{F}\left(a_{\varphi_1}\ostar  a_{\varphi_2}\right)\right)\overline{\kappa_\alpha}=\left(I_{\overline{V_\alpha}}\otimes\mathcal{F}\left(a_{\varphi_1\star\varphi_2}\right)\right)\overline{\kappa_\alpha}
 \\&= \widehat{a_{\varphi_1\star \varphi_2}}(\alpha)=\widehat{a_{\varphi_1}}\left(\alpha\right)\circ\widehat{a_{\varphi_2}}\left(\alpha\right).
 \end{align*} The second equality uses Van Daele's Convolution Theorem \ref{VDCT} and final equality is Proposition \ref{conv1} $\bullet$
 \end{proof}

\bigskip

\begin{lemma}\label{lemmax}
Where the sum is over unitary irreducible representations,
$$\widehat{h}(\varphi_1\star \varphi_2)=\sum_{\alpha\in\text{Irr}(G)}{d_\alpha}\operatorname{ Tr}\left(\widehat{a_{\varphi_1}}\left(\alpha\right)\widehat{a_{\varphi_2}}\left(\alpha\right)\right)\,.\,\,\qquad (\varphi_1,\,\varphi_2\in F(\widehat{G}))$$
\end{lemma}
\begin{proof}
The proof uses the convolution theorem of Van Daele and the definition of $\widehat{h}$ to find
$$\widehat{h}\left(\mathcal{F}\left(a_{\varphi_1}\right)\star\mathcal{F}\left(a_{\varphi_2}\right)\right)=\widehat{h}\left(\mathcal{F}\left(a_{\varphi_1}\ostar a_{\varphi_2}\right)\right)=\varepsilon\left(a_{\varphi_1}\ostar  a_{\varphi_2}\right).$$
Now use the Diaconis--Van Daele Inversion Theorem \ref{DVDIT} and the Diaconis--Van Daele Convolution Theorem \ref{conv1}
\begin{align*}
\varepsilon\left(a_{\varphi_1}\ostar  a_{\varphi_2}\right)&=\sum_{\alpha\in\text{Irr}(G)}d_\alpha\text{ Tr}\left(\widehat{a_{\varphi_1}\ostar  a_{\varphi_2}}\left(\alpha\right)\right)=\sum_{\alpha\in\text{Irr}(G)}d_\alpha\text{ Tr}\left(\widehat{a_{\varphi_1}}\left({\alpha}\right)\widehat{a_{\varphi_2}}\left({\alpha}\right)\right)\qquad\bullet
\end{align*}
\end{proof}

\begin{proposition}
Suppose that $\nu\in M_p(G)$. If $\kappa_\tau$ is the trivial representation, $\lambda\mapsto \lambda\otimes \mathds{1}_{G}$,  then $\widehat{a_\nu}\left(\tau\right)=I_{\overline{\mathbb{C}}}$.
\end{proposition}
\begin{proof} A calculation:
\begin{align*}
\widehat{a_\nu}\left(\tau\right)\lambda=\left(I_{\overline{\mathbb{C}}}\otimes\nu\right)\overline{\kappa_{\tau}}\left(\lambda\right)=\left(I_{\overline{\mathbb{C}}}\otimes \nu\right)(\lambda\otimes \mathds{1}_{G}^*)=\lambda\otimes1\equiv\lambda\qquad\bullet
\end{align*}
\end{proof}

\begin{proposition}
Suppose that $\kappa_{\alpha}$ is a non-trivial and irreducible unitary representation, then $\widehat{\mathds{1}_{G}}\left(\alpha\right)=0$.\label{vanish}
\end{proposition}
\begin{proof}
Another calculation:
\begin{align*}
\widehat{\mathds{1}_{G}}\left(\alpha\right)\overline{e_j}&=\left(I\otimes\mathcal{F}\left(\mathds{1}_{G}\right)\right)\sum_i\overline{e_i}\otimes\left(\rho_{ij}^\alpha\right)^*=\sum_i\overline{e_i}\,\mathcal{F}\left(\mathds{1}_{G}\right)\left(\rho_{ij}^\alpha\right)^*
=\sum_{i}\int_{G}\left(\left(\rho_{ij}^\alpha\right)^*\mathds{1}_{G}\right)\cdot \overline{e_i}.
\end{align*}
Note that $\mathds{1}_{G}$ is the matrix element of the trivial representation and $\alpha$ is not equivalent to the trivial representation. Therefore, by the orthogonality relation $\widehat{\mathds{1}_{G}}\left(\alpha\right)=0$ $\bullet$
\end{proof}

\bigskip

\subsection{The Upper Bound Lemma}
At this point, to mirror the classical notation, use $\pi$ for $\int_G$. Recall, from Section \ref{dist}, that
$$\|\nu-\pi\|^2\leq \frac14\, \widehat{h}\left((\nu-\pi)^*\star (\nu-\pi)\right)\,.\,\,\qquad (\nu,\,\mu\in F(\widehat{G}))$$

\bigskip

\begin{lemma}
(Diaconis--Shahshahani Upper Bound Lemma)
Let $\nu\in M_p(G)$. Then
\begin{align*}
\|\nu^{\star k}-\pi\|^2\leq \frac14\sum_{\alpha\in\text{Irr}(G)\backslash\{\tau\}}d_\alpha\operatorname{Tr} \left[\left(\widehat{\nu}\left(\alpha\right)^*\right)^k\left(\widehat{\nu}\left(\alpha\right)\right)^k\right],
\end{align*}
where the sum is over all non-trivial irreducible unitary representations.
\end{lemma}
\begin{proof}
Noting that $a_\pi=\mathds{1}_G$ write
\begin{align*}
\widehat{h}\left((\nu-\pi)^*\star (\nu-\pi)\right)&=\widehat{h}\left(\mathcal{F}\left(a_\nu-\mathds{1}_{G}\right)^*\star\mathcal{F}\left(a_\nu-\mathds{1}_{G}\right)\right).
\end{align*}
Now using Lemma \ref{lemmax} and (\ref{star}), this is equal to
\begin{align*}
\widehat{h}\left((\nu-\pi)^*\star (\nu-\pi)\right)&=\sum_{\alpha\in\operatorname{Irr(G)}}d_{\alpha}\operatorname{Tr} \left[\widehat{\left(a_\nu-\mathds{1}_{G}\right)}\left(\alpha\right)^*\widehat{\left(a_\nu-\mathds{1}_{G}\right)}\left(\alpha\right)\right].
\end{align*}
Now note that
\begin{align*}\widehat{\left(a_\nu-\mathds{1}_{G}\right)}\left(\alpha\right)&=\widehat{a_\nu}\left(\alpha\right)-\widehat{\mathds{1}_{G}}\left(\alpha\right).
\end{align*}
If $\alpha=\tau$, the trivial representation, then this yields zero as both terms are the identity on $\overline{\mathbb{C}}$. If $\alpha$ is non-trivial, then $\widehat{\mathds{1}_{G}}\left(\alpha\right)=0$ and thus (using the notation $\widehat{a_\nu}(\alpha)=\widehat{\nu}(\alpha)$):
$$\|\nu-\pi\|^2\leq \frac14\,\widehat{h}\left((\nu-\pi)^*\star (\nu-\pi)\right) =\frac14\sum_{\alpha\in\text{Irr}(G)\backslash\{\tau\}}d_\alpha\operatorname{Tr} \left[\widehat{\nu}\left(\alpha\right)^*\widehat{\nu}\left(\alpha\right)\right].$$
Apply the Diaconis--Van Daele Convolution Theorem \ref{DVDCT} $k$ times $\bullet$
\end{proof}

\bigskip

Note  this is \emph{exactly} the same as the classical Diaconis--Shahshahani Upper Bound Lemma.

\bigskip

\subsection{Lower Bounds}

\begin{lemma}\label{lbl}
(Lower Bound Lemma)
Suppose that $\nu\in M_p(G)$ and $\rho$ the matrix element of a non-trivial one dimensional representation. Then
\begin{align*}
\|\nu^{\star k}-\pi \|\geq \frac12 |\nu(\rho)|^k.
\end{align*}

\end{lemma}
\begin{proof}
Starting with (\ref{lb}), note that the argument from Proposition \ref{vanish} shows that $\rho$ has zero expectation under the random distribution. Note also that $\rho$ is unitary (Proposition 3.1.7 v., \cite{Timm}), thus norm one and thus  a suitable test function.

\bigskip

Note that for a one dimensional representation, by the Convolution Theorem \ref{DVDCT}:
$$|\nu^{\star k}(\rho)|=\left|\overline{\nu^{\star k}(\rho)}\right|=\left|\widehat{\nu^{\star k}}(\rho)\right|=|\widehat{\nu}(\rho)^k|=\left|\overline{\nu(\rho)}\right|^k=|\nu(\rho)|^{k}\,\,\,\bullet$$
\end{proof}
\section{Examples}
\subsection{Random Walks on Dual Groups}
Consider the dual quantum group, $\widehat{G}$, of a not-necessarily abelian group, $G$, given by $F(\widehat{G}):=\mathbb{C}G$. States on $F(\widehat{G})$ are given by positive definite functions $u\in F(G)=\mathbb{C}\widehat{G}$ (see Bekka, de la Harpe and Valette (Proposition C.4.2., \cite{Bekka})). Furthermore, there is a correspondence between positive definite functions and unitary representations on $G$ together with a vector. In particular, for each positive definite function $u$ there exists a unitary representation $\rho:G\rightarrow \operatorname{GL}(V)$ and a  vector $\xi\in V$ such that
\begin{equation} u(s)=\langle\rho(s)\xi,\xi\rangle,\label{posdef}\end{equation}
and for each unitary representation and  vector (\ref{posdef}) defines a positive definite function on $G$. For $u$ to be a state it is required that $u(e)=1$ and so $\langle \xi,\xi\rangle=1$; i.e. $\xi$ is a unit vector. Therefore probabilities on $\widehat{G}$ can be chosen by selecting a given representation and unit vector.

\bigskip

Since $\Delta(\delta^s)=\delta^s\otimes\delta^s$, it follows that $\kappa_s(\lambda)=\lambda\otimes \delta^s$ defines a (co)representation (with $\tau=\kappa_e$), and thus all irreducible representations are of this form by counting. This makes the application of the upper bound lemma straightforward. Let $u\in M_p(\widehat{G})$ so that $\widehat{u}(\kappa_s)=\overline{u(s)}$
and so
$$\widehat{u}(\kappa_s)^*\widehat{u}(\kappa_s)=|u(s)|^2.$$
Therefore the upper bound lemma yields:
$$\|u^{\star k}-\pi\|^2\leq \frac{1}{4}\sum_{t\in G\backslash \{e\}}|u(t)|^{2k}.$$
\subsubsection{A Walk on $\widehat{S_n}$}
Consider  the quantum group $\widehat{S_n}$ (given by $F(\widehat{S_n}):=\mathbb{C} S_n$) with a state $u\in M_p(\widehat{S_n})$ given by the permutation representation on $\mathbb{C}^n$ given by $\pi(\sigma)(e_i)=e_{\sigma(i)}$ together with the unit vector $\xi$ with components
$$\alpha_i=\sqrt{n^{n-i}\frac{n-1}{n^n-1}}.$$
This vector is such that the $\alpha_i$ are in geometric progression with common ratio $\displaystyle\frac{1}{\sqrt{n}}$. For large $n$, this vector is approximately given by:
$$\xi\approx \left(1,\frac{1}{\sqrt{n}},\frac{1}{\sqrt{n^2}},\cdots \right)\approx (1,0,0,\cdots).$$
Following this through
$$u(\sigma)\approx\begin{cases}
                    1 & \mbox{if } \sigma(1)=1, \\
                    0 & \mbox{otherwise}.
                  \end{cases}$$
                  \newpage
To establish the upper bound some elementary inequalities will be used.
\begin{lemma}
The following hold for $n\geq 5$:
\begin{align}
  \left(\frac{4}{n}\right)^k(n-1) & \leq 1, \text{ for }k>n^{n-1}\ln(n)/2,\,  \label{ineq1}\\
  \sqrt{\ln(n-1)}-n+2 & \leq 0, \label{ineq2}\\
  A_n=\frac{(n-1)(\sqrt{n}-1)^2}{n^n-1}n^{n-2}&\rightarrow 1^-\text{ monotonically}, \label{ineq3} \\
  g(n)=\frac{(\sqrt{n}-1)^2}{n^n-1}n^{n-1} & \rightarrow 1^-\text{ monotonically}. \label{ineq4}
\end{align}
\end{lemma}
\begin{proof}
For (\ref{ineq1}), note that $k>n^2\ln(n)/2$ and thus $\displaystyle \left(\frac{4}{n}\right)^k(n-1)\leq n\left(\frac{4}{n}\right)^{n^2\ln(n)/2}:=h(n)$. Note that $h(5)<1$ and
      $$\frac{d}{dn}\left(\ln(h(n))\right)=n^2\ln(n)\left[\ln\left(\frac{4}{n}\right)\right]+n^2\left[\ln\left(\frac{2e^{1/n^2}}{n}\right)\right]<0,$$
while (\ref{ineq2}) is trivial. Note that $A_n\rightarrow 1$ and
  \begin{align*}
  \frac{(n^n-1)^2}{n^{n-3}(\sqrt{n}-1)}\frac{dA_n}{dn}&=n^n(n-2)+n^2\sqrt{n}(n^{n-2}-(\ln n+1))+n\ln n(n-1)
  \\ &\quad+n(n-2)+\sqrt{n}(n\ln n-1)+n\sqrt{n}+2>0
  \end{align*}
  for $n\geq 3$, and so the convergence is monotonic. Also $g(n)\rightarrow 1$ and
  \begin{align*}
    \frac{(n^n-1)^2}{n^{n-5/2}(\sqrt{n}-1)}g'(n) & =n^2(n^{n-2}\sqrt{n}-(\ln n+1))+\sqrt{n}(n\,\ln n+n-1)>0
  \end{align*}
  for $n\geq 3$ $\bullet$
\end{proof}
\subsubsection*{Upper Bounds}
\emph{For $k=\alpha\,n^{n-1}\ln(n)/2+c\,n^{n-2}$, $\alpha>1$, and $n$ large}
$$\|u^{\star k}-\pi\|^2\leq\frac12 e^{-2Ac}$$
\emph{where $A$ can be chosen to be any $A<1$ for $n$ sufficiently large.}
\begin{proof}
Note that
\begin{align*}
u(\sigma)=\langle\xi,\pi(\sigma)\xi\rangle=\frac{n^{n+1}-n^n}{n^n-1}\sum_{i=1}^n\frac{1}{\sqrt{n^{i+\sigma(i)}}}.
\end{align*}
Therefore, using the Upper Bound Lemma,
\begin{align*}
\|u^{\star k}-\pi\|^2&\leq\frac{1}{4}\left(\frac{n^{n+1}-n^n}{n^n-1}\right)^{2k}\sum_{\sigma\in S_n\backslash \{e\}}\left[\sum_{i=1}^n\frac{1}{\sqrt{n^{i+\sigma(i)}}}\right]^{2k}.
\end{align*}
Define for $a_i=1/\sqrt{n^i}$
$$S(\sigma)=\sum_{i=1}^na_ia_{\sigma(i)}.$$
An \emph{inversion} is an ordered pair $(j,k)$ with $j,k\in \{1,\dots,n\}$ with $j<k$ and $\sigma(j)>\sigma(k)$. All non-identity permutations have at least one inversion.

\bigskip

Taking the approach of Steele (\cite{perm}, P. 79), for any inversion $(j,k)$   define a new permutation by

$$\tau_1(i):=(j\quad k)\sigma(i)=\begin{cases}\sigma(i)&\text{ if }i\neq j,k,\\ \sigma(j)&\text{ if } i=k,\\ \sigma(k)&\text{ if }i=j.\end{cases}$$

Steele shows that
$$S(\sigma)\leq S(\tau_1).$$
That is, multiplying by $(j\quad k)$, whenever $(j,k)$ is an inversion,  in this fashion, increases the number of fixed points, reduces the number of inversions and increases $S$. This can always be done until $\tau_r=e$.  If the maximising $\sigma\in S_{n}/\{e\}$ were not a transposition, then it would be the product of at least two transpositions. Take one of the transpositions $(j\quad k)$: it is certainly an inversion. By the referenced calculation, $\tau_i:=(j\quad k)\sigma$ has $S(\tau_i)\geq S(\sigma)$.  Therefore, no matter what the starting permutation $\sigma$, $\displaystyle\tau_{r-1}$ is a transposition and therefore to maximise $S$ on $S_n\backslash\{e\}$ one just maximises over transpositions.

\bigskip

The one-step differences between the $a_i$ is decreasing in $i$ and so the smallest one-step difference is between $a_{n-1}$ and $a_n$. Let $(j\quad k)$ be a transposition with $j<k$. Note that
\begin{align*}
S((n-1\quad n))-S((j\quad k))&=a_j^2+a_k^2+2a_{n-1}a_n
\\&-a_{n-1}^2-a_n^2-2a_ja_k
\\&=(a_j-a_k)^2-(a_n-a_{n-1})^2\geq 0,
\end{align*}
so that $S$ is maximised at $(n-1\quad n)$.

\bigskip

Partition $S_n\backslash \{e\}$ into $F_1$ and $F_1^C$ where $F_1$ is the set of permutations with $\sigma(1)=1$. On $F_1$,
\begin{align*}
S(\sigma)\leq S((n-1\quad n))&=\frac{1}{n^n}\frac{n^n-n^2}{n-1}+\frac{2\sqrt{n}}{n^n}=:f_0.
\end{align*}
Now consider the maximum of $S$ on $F_1^C$. From Steele's argument \cite{perm}, it is known that strictly increasing the number of fixed points (by multiplying by suitably chosen transpositions), increases $S$.  Also, if written in the disjoint cycle notation, elements of $F_1^C$ must contain a cycle of the form $(1\quad i_2 \quad \dots \quad i_N)$. By multiplying by suitably chosen transpositions, any disjoint cycle not containing $1$ may be factored out whilst increasing $S$. Then write
$$(1\quad i_2 \quad \dots \quad i_N)=(1\quad i_N)(1\quad i_{N-1})\cdots (1\quad i_2),$$
so that the maximum of $S$ on $F_1^C$ occurs at an element of the form
$$\sigma=\prod_{k=N}^2(1\quad i_k).$$
%\newpage
All transpositions are inversions therefore can be removed --- all the time increasing $S$ --- until one gets a transposition of the form $(1\quad i)$. The maximum must occur at such a transposition. Note that
\begin{align*}
S((1\quad 2))-S((1\quad i))&=2a_1a_2+a_i^2-2a_1a_i-a_2^2
\\&=(a_1-a_i)^2-(a_1-a_2)^2\geq 0.
\end{align*}
Therefore $S(1\quad i)\leq S(1\quad 2)$ and for any $\sigma\in F_1^C$:
\begin{align*}
S(\sigma)&\leq S(1\quad 2)=\frac{2\sqrt{n}}{n^2}+\frac{1}{n^2}\frac{1}{n^n}\frac{n^n-n^2}{n-1}=:f_1.
\end{align*}
For $n\geq 4$, $f_1\leq 2f_0/\sqrt{n}$ as
\begin{align*}
\frac{2}{\sqrt{n}}f_0-f_1&=\frac{1}{n^2n^n(n-1)}[n^3\sqrt{n}((n^{n-3}-2)+n^{n-4}(n-\sqrt{n}))+n^2(4n-3)]\geq0.
\end{align*}

\bigskip

Therefore the Upper Bound Lemma yields:
\begin{align*}
\|u^{\star k}-\pi\|^2 &\leq \frac{1}{4}\left(\frac{n^{n+1}-n^n}{n^n-1}\right)^{2k}\left(\sum_{\sigma\in F_1}S(\sigma)^{2k}+\sum_{\sigma\in F_1^C}S(\sigma)^{2k}\right).
\end{align*}
It is shown in the author's PhD thesis (\cite{PhD}, p. 93) that this yields:
\begin{align}
\|u^{\star k}-\pi\|^2 &\leq \frac14 \left(1-\frac{(n-1)(\sqrt{n}-1)^2}{n^n-1}\right)^{2k}\sqrt{n-1}(n-1)^{n-1}e^{2-n}\left[1+\left(\frac{4}{n}\right)^k(n-1)\right].\label{SnUB}
\end{align}
Using (\ref{ineq1}), and  $(1-x)\leq \exp(-x)$, a rewriting yields:
\begin{align*}
  \|u^{\star k}-\pi\|^2 & \leq \frac12 \exp\left(-2k\,\frac{(n-1)(\sqrt{n}-1)^2}{n^n-1}+(n-1)\ln(n-1)\right)\times \\
   &\qquad \exp\left(\ln(\sqrt{n-1})-n+2\right)
\end{align*}
By (\ref{ineq2}) the second exponential is less than one. Writing $k=\alpha\,n^{n-1}\ln(n)/2+c\,n^{n-2}$ and rewriting again
\begin{align*}
\|u^{\star k}-\pi\|&\leq \frac12 \exp\left(-2c\overbrace{\frac{(n-1)(\sqrt{n}-1)^2}{n^n-1}n^{n-2}}^{=A_n}\right)\times
\\ &\qquad \exp\left((n-1)\ln(n-1)-\alpha (n-1)\ln(n) \underbrace{\frac{(\sqrt{n}-1)^2}{n^n-1}n^{n-1}}_{=g(n)}\right)
\end{align*}
Note that for $\alpha>1$ and $n$ sufficiently large by (\ref{ineq4})
\begin{align*}
  (n-1)\ln(n-1)\left(1-\alpha \,g(n)\right) & \leq 0 \\
  \Rightarrow (n-1)\ln(n-1) & \leq \alpha\,(n-1)\ln(n-1)\,g(n) \\
  & \leq \alpha\,(n-1)\ln(n)\,g(n) \\
  \Rightarrow (n-1)\ln(n-1)-\alpha\,(n-1)\ln(n)\,g(n) & \leq 0,
\end{align*}
and the result follows $\bullet$
\end{proof}

\bigskip

If there were an effective lower bound for the total variation distance for \newline $k=n^{n-1}\ln(n)/2-c\,n^{n-2}$, and if such a bound approached one for large $c$, then the cut-off phenomenon would be exhibited for this random walk. Unfortunately no such result is at hand, but the following gives a partial result, showing that the random walk is still far from random for small multiples of $n^{n-2}$. The following will be used.

\begin{lemma}
The following hold for, respectively, $0<x<1/2$ and $n\geq 2$:
\begin{align}
  1-x & \geq e^{-x^2-x}, \label{ineq5}\\
  B_n=\frac{n^n}{n^n-1}\left(\frac{n^2}{n^n-1}+1\right) & \rightarrow 1^+\text{ monotonically}. \label{ineq6}
\end{align}
\end{lemma}
\begin{proof}
The bound (\ref{ineq5}) is (\cite{PhD}, Lemma 5.5.1) with $n=1$. Note that $B_n\rightarrow 1$, $B_2>1$, and note
$$\frac{(n^n-1)^3}{n^n}\frac{dB_n}{dn}=n^n[2n-(1+\ln n)(n^2+1)]-(1+\ln n)n^2-(2n-(1+\ln n))<0\qquad \bullet$$
\end{proof}

\subsubsection*{Lower Bounds}
\emph{For $k=\alpha\,n^{n-2}$, $\alpha>0$, and $n$ large}
$$\|u^{\star k}-\pi\|\geq\frac12 e^{-B\alpha}$$
\emph{where $B$ can be chosen to be any $B>1$ for sufficiently large $n$.}
\begin{proof}
Using the Lower Bound Lemma \ref{lbl} with the matrix element $\delta^{(n-1\quad n)}$, using a calculation from the PhD thesis (\cite{PhD}, p. 95), together with (\ref{ineq5}),
\begin{align*}
  \|u^{\star k}-\pi\| & \geq \frac12 \left|1-\frac{(n-1)(\sqrt{n}-1)^2}{n^n-1}\right|^{\alpha n^{n-2}} \\
   & \geq \frac12 \exp\left[\left(-\frac{(n-1)^2(\sqrt{n}-1)^4}{(n^n-1)^2}-\frac{(n-1)(\sqrt{n}-1)^2}{n^n-1}\right)\alpha n^{n-2}\right] \\
   & \geq \frac12 \exp\left[\left(-\frac{n^2n^n}{(n^n-1)^2}-\frac{n^n}{n^n-1}\right)\alpha \right]\\
   &=\frac12 \exp\left[\underbrace{\left(\frac{n^n}{n^n-1}\left(\frac{n^2}{n^n-1}+1\right)\right)}_{=B_n}\alpha \right]\qquad\bullet
\end{align*}

\end{proof}
\subsection{A Family of Walks on the Sekine Quantum Groups\label{Sek}}
\subsubsection{Sekine Quantum Groups}
Y. Sekine \cite{Sekine} introduced a family a finite quantum groups of order $2n^2$ that are neither commutative nor cocommutative.

\bigskip

The following follows the presentation of Franz and Skalski \cite{idempotent} rather than of Sekine. Let $n\geq 3$ be fixed and $\zeta_n=e^{2\pi i/n}$ and
$$\mathbb{Z}_n=\{0,1,\dots,n-1\},$$
 with addition modulo $n$.

 \bigskip

 Consider $n^2$ one-dimensional spaces $\mathbb{C} e_{(i,j)}$ spanned by elements indexed by $\mathbb{Z}_n\times\mathbb{Z}_n$, $\{e_{(i,j)}:i,j\in \mathbb{Z}_n\}$. Together with a copy of $M_n(\mathbb{C})$, spanned by elements $E_{ij}$ indexed by $\{(i,j)\,:\,i,j=1,\dots, n,\, 0\equiv n\}$, a direct sum of these $n^2+1$ spaces, the $2n^2$ dimensional space
 $$A_n=\left(\bigoplus_{i,j\in\mathbb{Z}_n}\mathbb{C} e_{(i,j)}\right)\oplus M_n(\mathbb{C}),$$
 can be given the structure of the algebra of functions on a finite quantum group denoted by $Y_n$ (so that $A_n=F(Y_n)$). On the one dimensional elements the comultiplication is given by, for $i,\,j\in \mathbb{Z}_n$:
 \begin{equation} \Delta(e_{(i,j)})=\sum_{\ell,m\in\mathbb{Z}_n}(e_{(\ell,m)}\otimes e_{(i-\ell,j-m)})+\frac{1}{n}\sum_{\ell,m=1}^n\left(\zeta_n^{i(\ell-m)}E_{\ell,m}\otimes E_{\ell+j,m+j}\right).\label{oneD}\end{equation}
 On the matrix elements in the $M_n(\mathbb{C})$ factor:
 \begin{equation}\Delta(E_{i,j})=\sum_{\ell,m\in\mathbb{Z}_n}(e_{(-\ell,-m)}\otimes \zeta_n^{\ell(i-j)}E_{i-m,j-m})+\sum_{\ell,m\in\mathbb{Z}_n}\left(\zeta_n^{\ell(j-i)}E_{i-m,j-m}\otimes e_{(\ell,m)}\right)\label{Mfact}\end{equation}
 The antipode is given by $S(e_{(i,j)})=e_{(-i,-j)}$ on the one dimensional factors and the transpose for the $M_n(\mathbb{C})$ factor. Sekine does not give the counit but by noting that $u_{(0,0)}=I_{n}$ (where $U\in M_n(M_n(\mathbb{C}))$ is defined in Sekine's original paper),  it can be seen  that the coefficient of the $e_{(0,0)}$ one-dimensional factor satisfies the counital property. The Haar state $\int_{Y_n}\in M_p(Y_n)$ is given by:
 $$\int_{Y_n}\left(\sum_{i,j\in \mathbb{Z}_n}x_{(i,j)}e_{(i,j)}+a\right)=\frac{1}{2n^2}\left(\sum_{i,j\in\mathbb{Z}_n}x_{(i,j)}+n\cdot \text{Tr}(a)\right).$$

 \bigskip

 Although Sekine restricts the construction to $n\geq 3$, for $n=1$ and $n=2$ the construction still satisfies the conditions of Kac and Paljutkin \cite{KP6} and so are algebras of functions of quantum groups. Sekine does not clarify but the construction for $n=2$ does not give the celebrated Kac--Paljutkin quantum group of order eight and indeed $Y_2$ is commonly mistaken for the Kac--Paljutkin quantum group in the literature. In fact, $Y_2$ is the dual group $\widehat{D_4}$ while $Y_1$ is the classical group $\mathbb{Z}_2$ \cite{PhD}.

\bigskip

To use the quantum Diaconis--Shahshahani Upper Bound Lemma, the representation theory of the quantum group must be well understood.
The representation theory of the Sekine quantum groups changes according to the parity of the parameter $n$ and the below restricts to $n$ odd.
\subsubsection{Representation Theory for $n$ Odd}
 For $n$ odd there are $2n$ one dimensional representations and $\binom{n}{2}$ two dimensional representations. Consider the convolution algebra $(F(Y_n),\star_{F(Y_n)})$. Sekine gives $2n$ minimal one-dimensional central projections, $\binom{n}{2}$ minimal two-dimensional central projections,  and matrix units in the two-dimensional subspaces. S\'{e}bastian Palcoux (private communication, March 2016) suggests a connection between projections and matrix units in the convolution algebra and the comultiplication in the algebra of functions. Palcoux's approach uses slightly different Fourier transforms and convolutions --- and the language of planar algebras (see \cite{Palc}) --- therefore his results could not be used directly. However it was possible to show that the one-dimensional central projections in $(F(Y_n),\star_{F(Y_n)})$ were the matrix elements of the one dimensional representations, while the two-dimensional central projections were, up to a factor of two, the matrix elements of the irreducible two-dimensional representations (see the appendix to \cite{PhD} for the brute-force verification). As far as the author knows this was not explicitly stated in the existing literature.

 \bigskip

  Let $\ell\in\{0,1,\dots,n-1\}$. Then
\begin{align*}
\rho_\ell^{\pm}=\sum_{i,j\in\mathbb{Z}_n}\zeta_n^{i\ell}e_{(i,j)}\pm\sum_{m=1}^n E_{m,m+\ell},
\end{align*}
are the $2n$ matrix elements of the one dimensional representations so that
$$\kappa_\ell^{\pm}(\lambda)=\lambda\otimes\rho_\ell^{\pm}\text{ and }\Delta(\rho_\ell^{\pm})=\rho_\ell^{\pm}\otimes \rho_\ell^{\pm}.$$
 Note that $\rho_0^+=\mathds{1}_{Y_n}$ is the matrix element of the trivial representation.

\bigskip

Now let $u\in\{0,1,\dots,n-1\}$ and $v\in\{1,2,\dots,(n-1)/2\}$. Each pair gives a two dimensional representation $\kappa^{u,v}:\mathbb{C}^2\rightarrow \mathbb{C}^2\otimes F(Y_n)$ with matrix elements:
\begin{align*}
\left(\begin{array}{cc}
\rho_{11}^{u,v} & \rho_{12}^{u,v}
\\ \rho_{21}^{u,v} & \rho_{22}^{u,v}
\end{array}\right)= \left(\begin{array}{cc}
\displaystyle \sum_{i,j\in\mathbb{Z}_n} \zeta_n^{iu+jv}e_{(i,j)} & \displaystyle \sum_{m=1}^n\zeta_n^{-mv}E_{m,m+u}
\\ \displaystyle\sum_{m=1}^n\zeta_n^{mv}E_{m,m+u} & \displaystyle\sum_{i,j\in \mathbb{Z}_n}\zeta_n^{iu-jv}e_{(i,j)}
\end{array}\right).
\end{align*}
\subsubsection{States on the Sekine Quantum Groups}
Consider the basis of  $F(\widehat{Y_n})$ dual to $\{e_{(i,j)}\,:\,i,j\in\mathbb{Z}_n\}\cup \{E_{i,j}\,:\,i,j=1,2,\dots,n\}$ given by
$$
\begin{array}{ccc}
e^{(i,j)}e_{(r,s)}=\delta_{i,r}\delta_{j,s}&\text{\qquad and \qquad} & e^{(i,j)}E_{r,s}=0,
\\ E^{i,j}e_{(r,s)}=0&\text{\qquad and \qquad}&E^{i,j}E_{r,s}=\delta_{i,r}\delta_{j,s}\end{array}.
$$
Let $\mu\in F(\widehat{Y_n})$:
$$\mu=\sum_{i,j\in\mathbb{Z}_n}x_{(i,j)}e^{(i,j)}+\sum_{p,q=1}^na_{p,q}E^{p,q}.$$
Franz and Skalski \cite{idempotent} show that $\mu\in M_p(Y_n)$ if and only if
\begin{itemize}
\item $x_{(i,j)}\geq0$ for all $i,j\in\mathbb{Z}_n$,
\item the matrix $A=(a_{pq})$ is positive,
\item $\operatorname{Tr}(\mu):=\sum_{i,j\in\mathbb{Z}_n}x_{(i,j)}+\sum_{p=1}^na_{p,p}=1$.
\end{itemize}
\subsubsection{A Random Walk}
Where $n$ is odd, and $J_n=\sum_{p,q=1}^n E^{p,q}\in M_n(\mathbb{C})$ the matrix of all ones, consider the state
$$ \nu=\frac{1}{8}(e^{(0,1)}+e^{(1,0)}+e^{(-1,0)}+e^{(0,-1)})+\frac{1}{2n}J_n\in M_p(Y_n).$$

The Diaconis--Shahshahani Upper Bound Lemma gives:
\begin{equation}\label{ublY}
  \|\nu^{\star k}-\pi\|^2\leq\frac14 \sum_{\alpha\in\operatorname{Irr}(Y_n)\backslash\{\tau\}}d_\alpha\operatorname{Tr}\left[\left(\widehat{\nu}(\alpha)^*\right)^k\widehat{\nu}(\alpha)^k\right]
\end{equation}

\bigskip

\begin{proposition}
For $k=\alpha\, n^2$, with $\alpha>1/20$  and  $n\geq 3$
$$\frac12 e^{-\alpha \pi^2/2}\leq \|\nu^{\star k}-\pi\|\leq c_ne^{-\alpha\pi^2/4},$$
with $c_n\rightarrow 1$ as $n\rightarrow \infty$.
\end{proposition}
\begin{proof}
For the Upper Bound, consider first the one-dimensional representations $\kappa_\ell^\pm$ (except the trivial representation $\kappa_0^+$).
\begin{align*}
\nu(\rho_\ell^{\pm})&=\frac{1}{8}\left(\zeta_n^\ell+\zeta_n^{-\ell}+2\right)\pm 2
\\  & =\begin{cases}
\dfrac12 \left(\cos^2\left(\dfrac{\pi \ell}{n}\right)+1\right) & \text{ if }\pm=+
\\[2ex] -\dfrac12\sin^2\left(\dfrac{\pi \ell}{n}\right) & \text{ if }\pm=-
\\[2ex] 0 & \text{ for }\kappa_\ell^\pm=\kappa_0^-
\end{cases}.
\end{align*}
Note that the $\kappa_\ell^+$ will dominate to such an extent that very trivial bounds may be used on the other terms.

\bigskip

The contribution to (\ref{ublY}) from the $\kappa_\ell^+$ is given by
\begin{align*}
\frac14 \sum_{\alpha=\kappa_\ell^+}d_\alpha\operatorname{Tr}\left[\left(\widehat{\nu}(\alpha)^*\right)^k\widehat{\nu}(\alpha)^k\right]&=\frac14 \sum_{\ell=1}^{n-1} |\nu(\rho_\ell^+)|^{2k}=\frac14 \sum_{\ell=1}^{n-1}\left(\dfrac12 \left(\cos^2\left(\dfrac{\pi \ell}{n}\right)+1\right)\right)^{2k}.
\end{align*}
Using the symmetry of $\cos^2(x)$ about $x=\pi/2$, $\cos x\leq e^{-x^2/2}$ on $[0,\pi/2]$ (Theorem 2, p. 26, \cite{DiaconisBook}), and employing the following inequality:
\begin{align*}
\frac{1}{2}(\cos^2x+1)\leq \frac12 (\cos x+1)&=\left(\cos\left(\frac{x}{2}\right)\right)^2\leq \left(e^{-(x/2)^2/2}\right)^2=e^{-x^2/4},
\\\Rightarrow \frac14 \sum_{\alpha=\kappa_\ell^+}d_\alpha\operatorname{Tr}\left[\left(\widehat{\nu}(\alpha)^*\right)^k\widehat{\nu}(\alpha)^k\right]&\leq
\frac12 \sum_{\ell=1}^{\frac{n-1}{2}} e^{-\pi^2\ell^2\alpha/2}
\end{align*}
A similar sum occurs in the analysis of the classical walk on $\mathbb{Z}_n$. See p.26 of Diaconis \cite{DiaconisBook} to see how such sums are handled (details teased out in Section 3.4 of \cite{MSc}):
$$\frac14 \sum_{\alpha=\kappa_\ell^+}d_\alpha\operatorname{Tr}\left[\left(\widehat{\nu}(\alpha)^*\right)^k\widehat{\nu}(\alpha)^k\right]\leq
e^{-\alpha\pi^2/2}$$
for $\alpha\geq 1/20$
\end{proof}

\bigskip

With the dominant term identified, the other terms can be bound crudely. In particular,
\begin{align*}
\frac14 \sum_{\alpha=\kappa_\ell^-}d_\alpha\operatorname{Tr}\left[\left(\widehat{\nu}(\alpha)^*\right)^k\widehat{\nu}(\alpha)^k\right]&=\frac{1}{4}\sum_{\ell=1}^{n-1}\frac{1}{4^{k}}\sin^{4k}\left(\frac{\pi \ell}{n}\right)\leq \frac{1}{4}\sum_{\ell=1}^{n-1}\frac{1}{4^k}=\frac{1}{4}\frac{n-1}{4^{\alpha n^2}}
\\&\leq \left(\frac{n}{4}\cdot \frac{e^{\alpha\pi^2/2}}{4^{\alpha n^2}}\right)e^{-\alpha \pi^2/2}\leq \left(\frac{n}{4}\cdot \frac{e^{\alpha\pi^2/2}}{e^{\alpha n^2}}\right)e^{-\alpha \pi^2/2}
\\&\leq \left(\frac{n}{4}\cdot e^{-\alpha(n^2-\pi^2/2)}\right)e^{-\alpha \pi^2/2}
\\&\underset{\alpha\geq 1/20}{\leq} \left(\frac{n}{4}\cdot e^{-(n^2-\pi^2/2)/20}\right)e^{-\alpha \pi^2/2}=:d(n)\cdot e^{-\alpha \pi^2/2}.
\end{align*}

\bigskip

Note that $\widehat{\nu}(\kappa^{u,v})$ is real and diagonal:
\begin{align*}
\widehat{\nu}(\kappa^{u,v})&=\frac{1}{2}\left(\cos\left(\dfrac{2\pi u}{n}\right)+\cos\left(\dfrac{2\pi v}{n}\right)\right)I_2
\\ \Rightarrow \operatorname{Tr}[(\widehat{\nu}(\kappa^{u,v})^*)^k\widehat{\nu}(\kappa^{u,v})^k]&=\frac{2}{4^k}\left(\cos\left(\frac{2\pi u}{n}\right)+\cos\left(\frac{2\pi v}{n}\right)\right)^{2k}
\end{align*}
Using the trivial $|\cos(x)+\cos(y)|\leq 2$ bound,
\begin{align*}
\frac14 \sum_{\alpha=\kappa^{u,v}}d_\alpha\operatorname{Tr}\left[\left(\widehat{\nu}(\alpha)^*\right)^k\widehat{\nu}(\alpha)^k\right]&\leq
\frac{1}{4^k}\sum_{\substack{u=0,\dots,n-1\\ v=1,\dots,\frac{n-1}{2}}}2^{2k}=\frac{2^{2k}}{4^{2k}}\frac{n(n-1)}{2}
\\&\leq \frac12\frac{n^2}{4^{\alpha n^2}}\leq \left(\frac{1}{2}n^2  e^{-\alpha(n^2-\pi^2/2)}\right)e^{-\alpha \pi^2/2}
\\&\leq 2n\cdot d(n)\cdot e^{-\alpha{\pi^2/2}}.
\end{align*}

\bigskip

Putting these bounds together gives the result with

$$c_n=\sqrt{1+d(n)+2n\cdot d(n)}$$

\bigskip

Consider the Lower Bound Lemma \ref{lbl}  with $\rho_{1}^+$:
$$\|\nu^{\star k}-\pi\|\geq \frac{1}{2}\left|\frac{\cos^2\left(\dfrac{\pi}{n}\right)+1}{2}\right|^k.$$
Consider $h(x)=\ln\left(\dfrac12 (\cos^2x+1)\,e^{x^2/2}\right)$. Note $h(0)=h'(0)=h''(0)=0$ but
$$h''(x)=\frac{(1-\cos^2x)(3-\cos^2x)}{(\cos^2x+1)^2}>0,$$
for $x\in (0,\pi)$. Therefore $\dfrac12 (\cos^2x+1)\geq e^{-x^2/2}$ and
\begin{align*}
\|\nu^{\star k}-\pi\|&\geq \frac12 e^{-\pi^2k/2n^2}=\frac{1}{2}e^{-\alpha \pi^2/2}\qquad \bullet
\end{align*}

\bigskip
%\newpage
Note that $c_n\rightarrow 1$ very rapidly. Below the bounds are plotted for $c_n=1$:

\begin{figure}[ht]\begin{center}\epsfig{figure=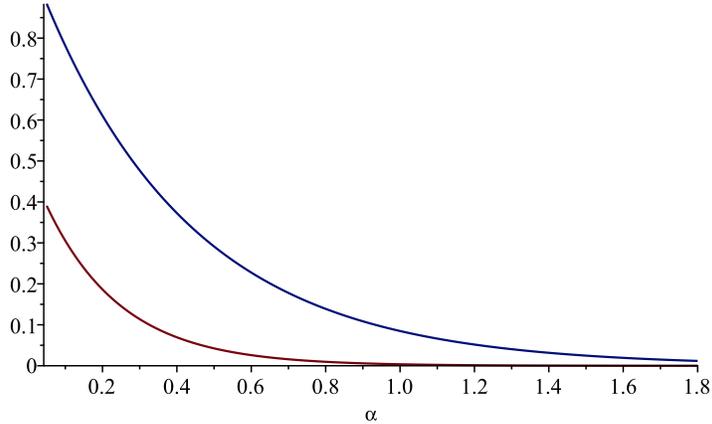,scale=0.5}\caption{Note that for $k=n^2/5$ the distance to random is bounded above by $3/5$ and bounded below away from zero. However at $k=3n^2/5$ the distance to random is still bounded away from zero. Therefore the cut-off phenomenon is not exhibited.}\end{center}\end{figure}

\newpage

\subsection*{Acknowledgements.} I would like to thank Adam Skalski; Example \ref{Sek} was developed during a visit to Adam at the Institute of Mathematics of the Polish Academy of Sciences (IMPAN), Warsaw, Poland. This trip was financially supported by IMPAN and also Cork Institute of Technology. I would like to thank Uwe Franz for assisting with Proposition \ref{tvd}. I would like to thank Amaury Freslon for encouragement and helpful comments; in particular for help in greatly improving the presentation of the bounds for the random walk on $\widehat{S_n}$.  The rest of the paper was developed during the author's PhD study at University College Cork, under the supervision of Stephen Wills.

\end{document}